\documentclass[11pt]{article}

\usepackage{amsfonts,amsmath,latexsym,color,epsfig}
\newtheorem{theorem}{Theorem}[section]

\newcommand{\qed}{\hfill \ensuremath{\Box}}

\newenvironment{proof}
      {\medskip\noindent{\bf Proof:}\hspace{1mm}}
      {\hfill$\Box$\medskip}

\input{epsf}

\makeatletter
\def\Ddots{\mathinner{\mkern1mu\raise\p@
\vbox{\kern7\p@\hbox{.}}\mkern2mu
\raise4\p@\hbox{.}\mkern2mu\raise7\p@\hbox{.}\mkern1mu}}
\makeatother

\title{\vspace{-0.7cm} Testing perfection is hard}
\author{Noga Alon\thanks{Sackler School of Mathematics and 
Blavatnik School of Computer Science, Tel Aviv University, Tel Aviv 
69978, Israel and Institute for Advanced Study, Princeton, 
New Jersey, 08540. 
Email: {\tt nogaa@tau.ac.il.} 
Research supported in part by an ERC Advanced grant, by a
USA-Israeli BSF
grant and by NSF grant No.  DMS-0835373.}
\and Jacob Fox\thanks{Department of Mathematics, MIT,
Cambridge, MA 02139-4307. Email: fox@math.mit.edu. Research
supported by a Simons Fellowship and NSF grant DMS 1069197.}}
\date{}
\begin{document}
\maketitle

\begin{abstract}
A graph property $\mathcal{P}$ is strongly testable if for every fixed
$\epsilon>0$ there is a one-sided $\epsilon$-tester for $\mathcal{P}$
whose query complexity is bounded by a function of $\epsilon$. In
classifying the strongly testable graph properties, the first author 
and Shapira showed
that any hereditary graph property (such as $\mathcal{P}$ the family of
perfect graphs) is strongly testable. A property is easily testable if
it is strongly testable with query complexity bounded by a polynomial
function of $\epsilon^{-1}$, and otherwise it is hard. One of our main
results shows that testing perfectness is hard. The proof shows that
testing perfectness is at least as hard as testing triangle-freeness,
which is hard. On the other hand, we show that  induced $P_3$-freeness
is easily testable. This settles one of the two exceptional graphs, the
other being $C_4$ (and its complement), left open in the characterization
by the first author and Shapira of graphs $H$ for which induced $H$-freeness is
easily testable.
\end{abstract}

\section{Introduction}

Property testing is an active area of computer science where one wishes to
quickly distinguish between objects that satisfy a property from objects
that are far from satisfying that property. The study of this notion
was initiated by Rubinfield and Sudan  \cite{RuSu}, and subsequently
Goldreich, Goldwasser, and Ron \cite{GGR} started the investigation
of property testers for combinatorial objects.  Graph property
testing in particular has attracted a great deal of attention.  A {\it
property} $\mathcal{P}$ is a family of (undirected) graphs closed under
isomorphism. A graph $G$ with $n$ vertices is {\it $\epsilon$-far from
satisfying $\mathcal{P}$} if one must add or delete at least $\epsilon
n^2$ edges in order to turn $G$ into a graph satisfying $\mathcal{P}$.

An {\it $\epsilon$-tester} for $\mathcal{P}$ is a randomized algorithm,
which given $n$ and the ability to check whether there is an edge between
a given pair of vertices, distinguishes with probability at least $2/3$
between the cases $G$ satisfies $\mathcal{P}$ and $G$ is $\epsilon$-far
from satisfying $\mathcal{P}$. Such an $\epsilon$-tester is one-sided
if, whenever $G$ satisfies $\mathcal{P}$, the $\epsilon$-tester
determines this with probability $1$.  A property $\mathcal{P}$ is
{\it strongly-testable} if for every fixed $\epsilon>0$ there exists a
one-sided $\epsilon$-tester for $\mathcal{P}$ whose query complexity is
bounded only by a function of $\epsilon$, which is independent of the
size of the input graph.

We call a property $\mathcal{P}$ {\it easily testable} if it is strongly
testable with a one-sided $\epsilon$-tester whose query complexity
is polynomial in $\epsilon^{-1}$, and otherwise $\mathcal{P}$ is
{\it hard}. This is analogous to classical complexity theory, where
an algorithm whose running time is polynomial in the input size is
considered fast, and otherwise slow. Call a hereditary graph property
{\it extendable} if for all but finitely many graphs in the family,
there is a larger graph in the family containing it as an induced
subgraph.  Most of the well-known hereditary graph properties are
extendable. As mentioned briefly in \cite{AFKS} and proved in
detail in \cite{GT}, there is a universal one-sided $\epsilon$-tester for
extendable hereditary graph properties which has query complexity at
most quadratic in the minimum possible query complexity of an optimal
one-sided $\epsilon$-tester. Indeed, it samples $d$ random vertices
(for some $d$), and if the subgraph they induce is in $\mathcal{P}$, it
accepts, and otherwise it rejects. The query complexity of this tester
is ${d \choose 2}$, and it is at least as accurate as any tester with query
complexity at most $d/2$. The query complexity is a lower bound for the
running time of an $\epsilon$-tester, and, if there is a polynomial time
recognition algorithm for membership in $\mathcal{P}$, the running time
is polynomial in the query complexity.  So while query complexity and
running time are different notions, they are often of comparable order.

For a graph $H$, let $\mathcal{P}_H$ denote the property of being
$H$-free, i.e., it is the family of graphs which do not contain $H$ as a
subgraph. The triangle removal lemma of Ruzsa and Szemer\'edi \cite{RuSz}
is one of the most influential applications of Szemer\'edi's regularity
lemma. It states that for every $\epsilon>0$ there is $\delta>0$ such
that any graph on $n$ vertices with at most $\delta n^3$ triangles can be
made triangle-free by removing at most $\epsilon n^2$ edges. The triangle
removal lemma is equivalent to the fact that $\mathcal{P}_{K_3}$ is strongly testable. Indeed,
the algorithm samples $t = 2\delta^{-1}$ triples of vertices uniformly at
random, where $\delta$ is picked according to the triangle removal lemma,
and accepts if none of them form a triangle, and otherwise rejects. Any
triangle-free graph is clearly accepted. If a graph is $\epsilon$-far from
being triangle-free, then it contains at least $\delta n^3$ triangles,
and the probability that none of the sampled triples forms a triangle
at most $(1 - \delta)^t < 1/3$. Notice that the query complexity depends
on the bound in the triangle removal lemma. As observed by Ruzsa and
Szemer\'edi, the triangle removal lemma gives a simple proof of Roth's
theorem \cite{Ro} that every dense subset of the integers contains a
$3$-term arithmetic progression. From Behrend's construction \cite{Be},
which gives a large subset of the first $n$ positive integers without
a $3$-term arithmetic progression, it follows that $\delta \leq \epsilon^{c\log \epsilon}$ in
the triangle removal lemma. This implies that testing triangle-freeness
is hard. Indeed, in the universal algorithm described earlier, in a
random sample of $d$ vertices, the expected number of triangles is
at most $\delta d^3$, and hence in the universal one-sided
$\epsilon$-tester for triangle-freeness, $1/3 \leq \delta d^3$, or
equivalently, $d \geq (3\delta)^{-1/3}$. As discussed earlier, the query
complexity of any one-sided  $\epsilon$-tester for triangle-freeness is
at least $d/2$.

The triangle removal lemma was extended in \cite{AFKS} (see also
\cite{ADLRY}) 
%by Erd\H{o}s, Frankl, and R\"odl \cite{EFR} 
to the graph removal lemma. It says that for each $\epsilon>0$
and graph $H$ on $h$ vertices there is $\delta=\delta(\epsilon,H)>0$ such
that every graph on $n$ vertices with at most $\delta n^h$ copies of $H$
can be made $H$-free by removing at most $\epsilon n^2$ edges. The graph
removal lemma similarly implies that testing $H$-freeness is strongly
testable. The proof, which uses Szemer\'edi's regularity lemma, gives
a bound on the query complexity which is a tower of height a power of
$\epsilon^{-1}$. This was somewhat improved recently by the second author
\cite{Fo} to a tower of height logarithmic in $\epsilon^{-1}$. The first
author \cite{Al} showed that $H$-freeness is easily testable if and only
if $H$ is bipartite.

For a graph $H$, let $\mathcal{P}_H^*$ denote the property of being
induced $H$-free, i.e., it is the family of graphs which do not contain
$H$ as an induced subgraph. The graph removal lemma was 
extended by the first author,
Fischer, Krivelevich, Szegedy \cite{AFKS} to the induced graph removal
lemma, which states that for every $\epsilon>0$ and graph $H$ on $h$
vertices there is $\delta>0$ such that any graph on $n$ vertices with
at most $\delta n^h$ induced copies of $H$ can be made induced $H$-free
by adding or removing at most $\epsilon n^2$ edges. The induced graph
removal lemma is equivalent to the fact that, for any graph $H$, 
the property $\mathcal{P}_H^*$ is
strongly testable. The proof, which uses a strengthening of Szemer\'edi's
regularity lemma, gives a bound on the query complexity which is wowzer
of height a power of $\epsilon^{-1}$, which is one higher in the Ackermann
hierarchy than the tower function. This has recently been improved
by Conlon and the second author \cite{CF} to the tower function.

The length of a path is the number of edges it contains, and we let $P_k$
denote the path of length $k$. The first author and Shapira \cite{AlSh06} showed
that for any graph $H$ other than the paths of length at most $3$, a
cycle of length $4$, and their complements, testing induced $H$-freeness
is hard. For $H$ a path of length at most $2$ or their complements,
induced $H$-freeness is easily testable. They left open the cases that
$H$ is a path of length $3$ or a cycle of length $4$ (and equivalently
its complement). Here we settle one of the two remaining cases.

\begin{theorem}\label{thmp3}
Induced $P_3$-freeness is easily testable. 
\end{theorem}

A well-known result of Seinsche \cite{Sei} gives a simple structure
theorem for induced $P_3$-free graphs. These graphs, also known as
cographs, are generated from the single vertex graph by complementation
and disjoint union. This is equivalent to the statement that every
induced $P_3$-free graph or its complement is not connected.

A quite general result of the 
first author and Shapira \cite{AlSh08a} states that
every hereditary family $\mathcal{P}$ of graphs is strongly testable. They
further asked which hereditary graph properties are easily testable, and,
in particular, for a few of the well-known hereditary families of graphs,
including perfect graphs and comparability graphs.

Note that the chromatic number of a graph is at least its clique number
as the vertices of any clique must receive different colors in a proper
coloring. A graph is {\it perfect} if every induced subgraph of it
satisfies that its clique number and chromatic number are equal. The
study of perfect graphs was started by Berge, partly motivated by the
study of the Shannon capacity in information theory, which lies between
the clique number and chromatic number of a graph. Perfect graphs
form a relatively large class of graphs for which several fundamental
algorithmic problems which are known to be NP-hard for general graphs,
such as the graph coloring problem, the maximum clique problem, and the
maximum independent set problem, can all be solved in polynomial time
(see \cite{GLS}). Also, it has significant connections with the study
of linear and integer programming (see, e.g., \cite{RAR}).

A famous conjecture of Berge, which was proved a few years ago by
Chudnovsky, Robertson, Seymour and Thomas \cite{CRST}, states that
a graph is perfect if and only if it contains no induced odd cycle
of length at least five or the complement of one. The proof in fact
establishes a stronger structural theorem for perfect graphs which was
conjectured by Conforti, Cornu\'ejols, and Vu\v skovi\' c. It says that
every perfect graph falls into one of a few basic classes, or admits one
of a few kinds of special decompositions.  Shortly afterwards, a proof
that perfect graphs can be recognized in polynomial time (as a function
of the number of vertices of the graph) was discovered by Chudnovsky,
Cornu\'ejols, Liu, Seymour, and Vu\v skovi\' c \cite{CCLSV}.

Another well-studied hereditary family of graphs are comparability
graphs. A comparability graph is a graph that connects pairs of elements
that are comparable to each other in a partial order. Gallai \cite{Gal}
classified these graphs by forbidden induced subgraphs, and Dilworth's
theorem \cite{Di} is equivalent to the statement that the complement of
comparability graphs are perfect. Further, comparability graphs can be
recognized in polynomial time (see McConnell and Spinrad \cite{MS}). Every
cograph is a comparability graph, and every comparability graph is a
perfect graph. It is natural to suspect that the structure theorem could
hint at a polynomial in $\epsilon^{-1}$ tester for perfectness similar to
testing cographs. However, we show that testing perfectness essentially
requires as much query complexity (or time) as testing triangle-freeness,
which is hard.

\begin{theorem}\label{perfectmain}
Testing perfectness is hard. 
\end{theorem} 

Indeed, Theorem \ref{C5compper} shows that from 
a graph on $n$ vertices which
is $14 \epsilon$-far from being triangle-free
but a random sample of $d$ vertices
is with probability at least $1/2$ triangle-free, we can construct a graph
on $5n$ vertices which is $\epsilon/25$-far from being induced $C_5$-free
but a random sample of $d$ vertices in it is a comparability graph  with
probability at least $1/2$. Since every comparability graph is perfect,
every perfect graph is induced $C_5$-free, and testing triangle-freeness
is hard, this implies the above theorem that testing perfectness is hard,
and further that testing for comparability graphs is hard.

\begin{theorem}\label{compmain}
Testing for comparability graphs is hard. 
\end{theorem} 

In the next section, we show that induced $P_3$-freeness is easily
testable. In Section \ref{perfsect}, we show that testing perfectness is
at least as hard as testing triangle-freeness, which is hard. We finish
with some concluding remarks. Throughout the paper, we systematically
omit floor and ceiling signs whenever they are not crucial for the sake
of clarity of presentation. We also do not make any serious attempt to
optimize absolute constants in our statements and proofs.

\section{Induced $P_3$-freeness is easily testable} 

A {\it cut} for a graph $G=(V,E)$ is a partition $V=V_1 \cup V_2$
into nonempty subsets such that there are no edges between $V_1$ and
$V_2$ or $V_1$ is complete to $V_2$. The following definition is a
natural relaxation of a cut. For $\beta>0$, define a {\it $\beta$-cut}
for a graph $G=(V,E)$ as a partition $V=V_1 \cup V_2$ into nonempty
subsets such that $e(V_1,V_2) \leq \beta|V_1||V_2|$ or $e(V_1,V_2) \geq
(1-\beta)|V_1||V_2|$. For a graph $G$ and vertex subset $S$, let $G[S]$
denote the induced subgraph of $G$ with vertex set $S$. Let $c(\beta,n)$
be the least $\delta$ for which there is a graph $G=(V,E)$ on $n$ vertices
which has no $\beta$-cut and has $\delta n^4$ induced copies of $P_3$.

\begin{theorem}
We have $c(\beta,n) \geq (\beta/100)^{12}$. 
\end{theorem}
\begin{proof}
Suppose for contradiction that there is a graph $G$ on $n$ vertices which
does not have a $\beta$-cut and has less than $\delta n^4$ induced copies
of $P_3$, where $\delta=(\beta/100)^{12}$. Since $G$ has no $\beta$-cut,
then $G$ contains an induced $P_3$. Hence, $1 \leq \delta n^4$ and $n
\geq \delta^{-1/4} \geq (100/\beta)^3$.

 Since $G$ has at most $\delta n^4$ induced copies of $P_3$, a random
 sample of $r=(8\delta)^{-1/4} \geq 10^5\beta^{-3}$ vertices has in
 expectation at most $\delta r^4=1/8$ induced copies of $P_3$. Hence,
 with probability at least $7/8$, a random sample of $r$ vertices contains
 no induced $P_3$.

Randomly sample a set $R=S \cup T$ of $r=s+t$ vertices from $V$, where
$s=t=r/2$. Let $E_0$ be the event that $G[R]$ is induced $P_3$-free,
so the probability of event $E_0$ is at least $7/8$.

Since $G$ does not have a $\beta$-cut, each vertex has more than
$\beta(n-1)$ neighbors and less than $(1-\beta)(n-1)$ neighbors. Let
$\alpha=\beta/2$. Hoeffding  (see Section 6 of \cite{Ho}) proved
that the hypergeometric distribution is at least as concentrated as
the corresponding binomial distribution. Thus, by the Azuma-Hoeffding
inequality (see, e.g., \cite{AlSp}), and the fact that each vertex $v \in
S$ has more than $\beta(n-1)$ neighbors, the probability that a particular
$v \in S$ has less than $\alpha (s-1)$ neighbors in $S$ is a most
$$e^{-\left((\beta-\alpha)(s-1)\right)^2/(2(s-1))}
=e^{-(\beta-\alpha)^2(s-1)/2}
\leq e^{-\beta^2 s/16} \leq \frac{1}{16s}.$$ Similarly, the probability
that $v$ has more than $(1-\alpha)(s-1)$ neighbors in $S$ is at most
$\frac{1}{16s}$. Let $E_1$ be the event that every vertex in $S$ has at
least $\alpha(s-1)$ and at most $(1-\alpha)(s-1)$ neighbors in $S$, i.e.,
the induced subgraph $G[S]$ has minimum degree at least $\alpha(s-1)$
and maximum degree at most $(1-\alpha)(s-1)$. By the union bound, the
probability of event $E_1$ is at least $1-2s \cdot \frac{1}{16s}=7/8$.

Let $U$ be the set of vertices $v \in V \setminus S$ which are complete
or empty to $S$. As the degree of each vertex of $G$ is at least $\beta(n-1)$ and at most $(1-\beta)(n-1)$, the probability that for a given vertex $v$, a random subset of $s$ vertices of $V \setminus \{v\}$ are all neighbors of $v$ or all nonneighbors of $v$ is at most $2(1-\beta)^s$. Hence, a given vertex has probability at most $2(1-\beta)^s$ of being in $U$. By linearity of expectation, 
the expected size of $U$ is at most $2(1-\beta)^s n$. Let $E_2$
be the event that $|U| \leq 16(1-\beta)^s n \leq 16e^{-\beta s}n \leq
\frac{\beta}{8}n$. By Markov's inequality, the probability of $E_2$
is at least $1-1/8=7/8$.

Let $E_3$ be the event that $T$ contains no vertex from $U$. By linearity of expectation, 
$\mathbb{E}[|U \cap T|] = \mathbb{E}[|U|]t/n \leq 2(1-\beta)^s t \leq 2e^{-\beta s}t \leq
\frac{1}{8}$. Therefore, event $E_3$ occurs with probability at least
$7/8$.

The probability that events $E_0$ and $E_1$ both occur is at least
$7/8-1/8=3/4$. If both of these events occur, then $G[S]$ has at least
one and at most $2^{\alpha^{-1}}$ cuts. Consider such a cut $S=S_1 \cup
S_2$ of $G[S]$, and suppose $S_1$ is complete to $S_2$ (the case $S_1$
is empty to $S_2$ can be treated similarly). For each such cut, consider
the partition $V \setminus S = U \cup V_0 \cup V_1 \cup V_2$  of vertices,
where $v \in V \setminus S$ satisfies $v \in V_0$ if $v \not \in U$ and it
is not complete to $S_1$ and not complete to $S_2$, 
$v \in V_1$ if it is complete to $S_2$
but not complete to $S_1$, and $v \in V_2$ if it is complete to $S_1$
but not to $S_2$.

Note that if $T$ contains a vertex from $V_0$, then the cut $S=S_1
\cup S_2$ of $G[S]$ does not extend to a cut of $G[R]$. If events $E_i$
for $i=0,1,2,3$ occur, which happens with probability at least $1/2$,
then $G[R]$ is induced $P_3$-free, so it has a cut, and no vertex in $T$
is complete or empty to $S$. In this case one of the cuts of $G[S]$
extends to a cut of $G[R]$, and hence, for at least one cut of $G[S]$, 
no vertex of $T$ is in the corresponding $V_0$.

We now condition on the occurrence of events $E_i$ for $i=0,1,2,3$. Note that
since the probability that this happens is at least $1/2$, for any other 
event $E$, the conditional probability that $E$ occurs given that
$E_i$ occur for $i=0,1,2,3$ is at most twice the probability of $E$
without any conditioning.

To complete the proof we claim that with positive probability 
$E_0,E_1,E_2,E_3$ occur and yet
the induced
subgraph on $S \cup T$ contains an induced $P_3$, contradicting
$E_0$. To do so  we apply the union bound over all cuts in $G[S]$
to show that with positive probability, for each such cut, either
$T$ contains a vertex of $V_0$ (and hence the cut cannot  be extended
to one in $G[R]$) or $T$ contains a vertex $v_1$ in $V_1$  
and a vertex $v_2$ in $V_2$,
which are nonadjacent, providing an induced $P_3$ in $G[S \cup T]$
on the vertices  
$v_1,v_2$ together with a
vertex $s_1 \in S_1$ not adjacent to $v_1$ and a vertex $s_2 \in S_2$
not adjacent to $v_2$.

We proceed with the proof of this claim. Conditioning on
$E_i$ for $i=0,1,2,3$, fix a cut $(S_1,S_2)$ in $G[S]$ and let
$V_0,V_1,V_2$ be as above. Consider two possible cases.
\vspace{0.3cm}

\noindent
Case 1: $|V_0| \geq \frac{2}{\alpha t} n$.

In this case, the probability that $T$  contains no vertex of $V_0$
is at most 
$$
(1-\frac{2}{\alpha t})^t \leq e^{-2/\alpha}<2^{-\alpha^{-1}-1},
$$ 
showing that even after our conditioning the probability
of this event is smaller than $2^{-\alpha^{-1}}$.
\vspace{0.3cm}

\noindent
Case 2: $|V_0| < \frac{2}{\alpha t} n \leq \frac{\beta}{8}n $.

Let $x=|U|+|V_0|$, $y=|S_1| + |V_1|$, and $z=|S_2| + |V_2|$, so
$x+y+z=n$. Assume without loss of generality that $y \leq z$. Since
the partition $V=(S_1 \cup V_1) \cup (S_2 \cup V_2 \cup U \cup V_0)$
is not a $\beta$-cut, there are at least $\beta y (z+x)$ missing edges
between these two sets. Since, in addition,  $S_1$ is complete to $S_2$,
$S_1$ is complete to $V_2$, and $V_1$ is complete to $S_2$, then these
missing edges go between $V_1$ and $V_2$ and between $S_1 \cup V_1$
and $U \cup V_0$.  Thus $$\frac{\beta}{2}yn \leq \beta y (z+x) \leq
|V_1||V_2|-e(V_1,V_2)+yx.$$ If events $E_i$ for $i=0,1,2,3$ occur, then $x
\leq \frac{\beta}{4}n$, and hence there are at least $\frac{\beta}{4}yn$
missing edges between $V_1$ and $V_2$.  In this case, every vertex of
$S_1$ is complete to $S_2 \cup V_2$, and hence $$(1-\beta)(n-1) \geq
z =n-x-y \geq n-\frac{\beta}{4}n-y$$ and  $$y \geq \frac{3\beta}{4}n-1
\geq \frac{\beta}{2}n.$$ Thus, the number of missing pairs between $V_1$
and $V_2$ in the case events $E_i$ for $i=0,1,2,3$ occur is at least
$\frac{\beta}{4}yn \geq \frac{\beta^2}{8}n^2$.

Let $E_4$ be the event that $T$ contains the two vertices of at least
one of the nonedges between $V_1$ and $V_2$. Given that there are at
least $\frac{\beta^2}{8}n^2$ edges missing between $V_1$ and $V_2$,
the probability that event $E_4$ occurs is at least the probability
that at least one of $t/2$ random 
pairs of vertices of $G$ contains one of the nonedges
between $V_1$ and $V_2$.  The probability that this does {\em not} happen
is at most
$$\left(1-\frac{\beta^2
n^2/8}{{n \choose 2}}\right)^{t/2} \leq e^{-\beta^2 t/8}
=e^{-\beta^2 10^5/(8 \cdot 2  \beta^3)}
=e^{-10^5 /(32 \alpha)}<2^{-\alpha^{-1}-1},
$$ 
and hence even after our conditioning the probability
of this event is smaller than $2^{-\alpha^{-1}}$.

By the union bound it now
follows that with positive probability $E_i$ for $i=0,1,2,3$ occur
and yet $G[S \cup T]$ contains an induced $P_3$. This is
a contradiction, completing the proof.  \end{proof}

Let $f(\epsilon,n)$ be the least $\delta$ for which there is a graph
$G=(V,E)$ on $n$ vertices which is $\epsilon$-far from being induced
$P_3$-free and has $\delta n^4$ induced copies of $P_3$.

\begin{theorem}\label{p3removal}
There is $n_0 \geq \epsilon n$ such that $f(\epsilon,n) 
\geq c(\epsilon,n_0)\epsilon^4 \geq (\epsilon/100)^{16}$. 
\end{theorem}
\begin{proof} 
Let $G=(V,E)$ be a graph on $n$ vertices which is $\epsilon$-far
from being induced $P_3$-free. Partition $V$ into two parts along an
$\epsilon$-cut, and continue refining parts along $\epsilon$-cuts of the
subgraphs induced by the parts until no part has an $\epsilon$-cut, and
let $V=V_1 \cup \ldots \cup V_k$ be the resulting partition. We modify
edges along these $\epsilon$-cuts to turn them into cuts, letting $G'$
be the resulting graph. The total fraction of pairs of vertices changed
in making $G'$ from $G$ is at most $\epsilon$, so at least $\epsilon
n^2-\epsilon{n \choose 2} \geq \epsilon n^2/2$ edges must be changed from
the resulting graph $G'$ to make it induced $P_3$-free. We can modify
edges in each $V_i$ to make it induced $P_3$-free, and the resulting graph
on $V$ is induced $P_3$-free. If $|V_i| \leq \epsilon n$ for $1 \leq i
\leq k$, then the number of edge modifications made to $G'$ to obtain an
induced $P_3$-free graph is at most $$\sum_{i=1}^k {|V_i| \choose 2} \leq
\frac{n}{2}\max_{1 \leq i \leq k} (|V_i|-1) < \frac{\epsilon n^2}{2},$$
 a contradiction. Thus, one of the parts $V_i$, call it $V_0$,
 has $n_0 > \epsilon n$
 vertices, and $G[V_0]$ has no $\beta$-cut. Therefore, the induced
 subgraph $G[V_0]$, and hence $G$, has at least $$c(\epsilon,n_0)n_0^4
 \geq c(\epsilon,n_0)\epsilon^4 n^4 \geq (\epsilon/100)^{12}\epsilon^4
 n^4 \geq (\epsilon/100)^{16}n^4$$ induced copies of $P_3$, completing
 the proof.
\end{proof}

Consider the following one-sided $\epsilon$-tester for induced
$P_3$-freeness. Let $\delta= (\epsilon/100)^{16}$. The algorithm samples
$t = 2\delta^{-1}$ quadruples of vertices uniformly at random,  and
accepts if none of them form an induced $P_3$, and otherwise rejects. Any
induced $P_3$-free graph is clearly accepted. If a graph is $\epsilon$-far
from being induced $P_3$-free, then it contains at least $\delta n^4$
induced $P_3$ by Theorem \ref{p3removal}, and the probability that none of
the sampled quadruples forms an induced $P_3$ is at most $(1 - \delta)^t <
1/3$. Note that the query complexity for this algorithm depends linearly
on $\delta^{-1}$, and hence polynomially on $\epsilon^{-1}$, completing
the proof of Theorem \ref{thmp3}. \qed

\section{Testing perfectness} \label{perfsect}

We first observe a couple of equivalent versions of the triangle removal
lemma. The {\it triangle edge cover number} $\nu(G)$ of a graph $G$
is the minimum number of edges of $G$ that cover all triangles in $G$,
i.e., it is the minimum number of edges of $G$ whose deletion makes
$G$ triangle-free. The triangle removal lemma thus says that for each
$\epsilon>0$ there is $\delta>0$ such that every graph on $n$ vertices
with at most $\delta n^3$ triangles satisfies $\nu(G) \leq \epsilon n^2$.

The {\it triangle packing number} $\tau(G)$ of a graph $G$ is the maximum
number of edge disjoint triangles in $G$. The following simple bounds hold
for all graphs: $$\tau(G) \leq \nu(G) \leq 3\tau(G).$$ Indeed, at least
one edge from each of the edge-disjoint triangles is needed in any edge
cover of the triangles in $G$, and deleting the $3\tau(G)$ edges from
a maximum collection of edge-disjoint triangles leaves a triangle-free
graph.  We remark that a well known conjecture of Tuza states that the
upper bound can be improved to $\nu(G) \leq 2\tau(G)$. Haxell \cite{Ha}
improved the upper bound factor to $3-\frac{3}{23}$.

Thus, up to a constant factor change in $\epsilon$, the triangle
removal lemma is the same as saying that a graph $G$ on $n$ vertices
with at least $\epsilon n^2$ edge disjoint triangles contains at least
$\delta n^3$ triangles. We can further suppose, up to a constant factor
change in $\epsilon$, that $G$ is tripartite. Indeed, every graph has
a tripartite subgraph which contains at least $2/9$ of the triangles
in a maximum collection of edge-disjoint triangles. This can be seen by
considering a uniform random tripartition. Each triangle has probability
$2/9$ of having one vertex in each part, so the expected number of the
edge-disjoint triangles in the tripartition is $2/9$ of the total, and
there is a tripartition for which the number of edge-disjoint triangles
is at least the expected number. We may thus assume $G$ is tripartite.

\begin{theorem}\label{C5compper}
Let $T$ be a graph on $n$ vertices which is $14\epsilon$-far from
being triangle-free such that a random sample of $d$ vertices of $T$ is
triangle-free with probability at least $1/2$. Then there is a graph $G$
on $5n$ vertices which is $\epsilon/25$-far from being induced $C_5$-free,
such that a random sample of $d$ vertices of $G$ is a comparability
graph with probability at least $1/2$.
\end{theorem}
\begin{proof}
By the remarks above, $T$ contains a tripartite subgraph $F$ which
contains at least $\frac{1}{3} \cdot \frac{2}{9} \cdot 14 \epsilon
n^2>\epsilon n^2=(\epsilon/25) (5n)^2$ edge-disjoint triangles. 
Denote the three parts of $F$ by $V_2,V_3,V_5$.  

Let $G=(V,E)$ be the graph on $5n$ vertices with partition $V=V_1 \cup
V_2 \cup V_3 \cup V_4 \cup V_5$, where $V_1$ and $V_4$ are of size $2n$ each,
and $V_2,V_3,V_5$ are the parts of $F$. We next specify the edges
between the various parts of $G$. Each part $V_i$, $1 \leq i \leq 5$,
is an independent set. There are no edges between $V_1$ and $V_2$,
between $V_1$ and $V_3$, between $V_3$ and $V_4$, and between $V_4$
and $V_5$. There is a complete bipartite graph between $V_1$ and $V_4$,
between $V_1$ and $V_5$, and between $V_2$ and $V_4$. The edges of $G$
are precisely the edges of $F$ between $V_2$ and $V_3$, and between
$V_3$ and $V_5$. Finally, between $V_2$ and $V_5$, the edges of $G$
are precisely the nonedges of $F$.

Arbitrarily order $T_1,\ldots,T_t$ a maximum collection of $t=\tau(F) \geq
\epsilon n^2$  edge-disjoint triangles in $F$. As $F$ is a tripartite
graph on $n$ vertices, $t=\tau(F)$ is at most the product of the two
smallest parts, which is at most $n^2/9$. For every triangle in $F$, the
same three vertices in $G$ with a vertex in $V_1$ and a vertex in $V_4$
form an induced $C_5$.  We next show that this implies that there are $t$
induced copies of $C_5$ in $G$, labeled $L_1,\ldots,L_t$, such that each
pair intersects in at most one vertex. In fact, we greedily construct
$L_1,\ldots,L_t$ so that they further satisfy that the vertex set of each
$L_i$ consists of the vertices of $T_i$ together with a vertex in $V_1$
and a vertex in $V_4$.

Suppose we have already constructed $L_j$ for $j<i$ satisfying the
desired properties. We next show how to construct $L_i$ with the
desired properties. Note that in a tripartite graph, the number of
edge-disjoint triangles containing a given vertex $v$ is at most the
minimum order of the two parts not containing $v$. It follows that
$T_i$ has nonempty intersection with at most $n$ of the $t$ triangles
$T_1,\ldots,T_t$. Hence, for $h=1,4$, at most $n$ vertices in $V_h$ are
in at least one $L_j$ with $j<i$ for which $T_j$ and $T_i$ share a vertex
in common. For $h=1,4$, delete these vertices from $V_h$, and denote
the resulting subset of $V_h$ as $V_h'$, so $|V_h'| \geq |V_h|-n =n$.
As $i-1<t<n^2 \leq |V_1'||V_4'|$, there is a pair $(v_1,v_4) \in V_1'
\times V_4'$ that is not in any $L_j$ with $j<i$. We pick $L_i$ to
be the induced $C_5$ in $G$ with vertices $v_1,v_4$ and the vertices
of $T_i$. It is clear from this construction that $L_i$ intersects
each $L_j$ with $j<i$ in at most one vertex. We therefore can greedily
construct the desired $t$ induced copies of $C_5$, and conclude that $G$
is $\epsilon/25$-far from being induced $C_5$-free.

On the other hand, the only triples $a<b<c$ of vertices in a linear
ordering which puts the vertices in $V_i$ before $V_j$ if $i<j$ with $a$
adjacent to $b$, $b$ adjacent to $c$, and $a$ not adjacent to $c$ are with
$a \in V_2$, $b \in V_3$, and $c \in V_5$ the vertices of a triangle
in $F$. Thus, by sampling $d$ vertices uniformly at random from $G$,
we sample at most $d$ vertices uniformly at random from $F$. These at
most $d$ vertices are triangle-free in $F$ with probability at least
$1/2$, and hence the $d$ random vertices in $G$ form a comparability
graph with probability at least $1/2$. This completes the proof.
\end{proof}

As discussed toward the end of the introduction, Theorem \ref{C5compper}
implies Theorem \ref{perfectmain}  that testing perfectness is hard,
and Theorem \ref{compmain} that testing for comparability graphs is hard.

A partially ordered set (poset) is a 
directed graph on a vertex set $P$ which
\begin{itemize} 
\item has no loops, i.e., no pair $(x,x)$ is an edge, 
\item has no antiparallel edges, 
i.e., if $(x,y)$ is an edge, then $(y,x)$ is not an edge, 
\item is transitive, i.e., if $(x,y)$ 
is an edge and $(y,z)$ is an edge, then $(x,z)$ is also an edge. 
\end{itemize} 
The fact that testing for posets is hard (at least as hard as testing for
triangle-freeness) follows from Theorem \ref{C5compper} by adding
directions. However, we next sketch a simpler proof. Let $T$ be a
tripartite graph on $n$ vertices with parts $V_1,V_2,V_3$ which is
$\epsilon$-far from being triangle-free. Consider the directed graph
$G$ on the same vertex set as $T$ with $(v_1,v_2) \in V_1 \times V_2$
an edge of $G$ if it is an edge of $T$, $(v_2,v_3) \in V_2 \times V_3$
an edge of $G$ if it is an edge of $T$, $(v_1,v_3) \in V_1 \times
V_3$ an edge of $G$ if it is not an edge of $T$, and there are no other edges. At least one pair in
every triangle of $T$ must be modified to turn $G$ into a poset, so $G$
is $\epsilon$-far from being a poset. Also, any subset of vertices which
is triangle-free in $T$ induces a poset in $G$. This implies that testing
for posets is at least as hard as testing for triangle-freeness.

\section{Concluding Remarks} 

We believe that comparing the number of queries needed to test various
properties, as done in this paper comparing testing perfectness and
triangle-freeness, could be an interesting direction for further
research. This is the analogue in property testing to the powerful
technique of hardness reductions in complexity theory. One general
class of hard graph properties for testing for which to compare with is
(not necessarily induced)
$H$-freeness for $H$ a fixed odd cycle.

We showed that testing perfectness is hard. This is equivalent to showing
that there is a graph which is $\epsilon$-far from being perfect such
that a random set of vertices of size polynomial in $\epsilon^{-1}$
is perfect with probability at least $1/2$. This still leaves the
possibility of getting a small witness if the graph is far from being
perfect. That is, does every graph which is $\epsilon$-far from being
perfect contain an induced odd cycle or its  complement of size at least $5$ 
and at most a
polynomial in $\epsilon^{-1}$?

We showed that testing induced $P_3$-freeness is easy, which is a step
toward completing the classification of graphs $H$ for which induced
$H$-free testing is easy. It remains to determine whether or not induced
$C_4$-freeness is easy.

Finally, it will be very interesting to characterize  
all easily testable graph properties. As all these 
properties have to be strongly testable, it follows from the main result
of \cite{AlSh08a}  that if
we restrict ourselves only to
natural properties, in the sense of \cite{AlSh08a}, 
then these properties have to be essentially hereditary.
Among the hereditary properties, 
properties that are known to be easily testable include the property
of being 
$k$-colorable for any fixed $k$, as shown in \cite{GGR}, as well
as a natural extension of it, as proved in \cite{GT}. As mentioned 
in the introduction, additional 
easily testable (hereditary) properties are $H$-freeness for
any bipartite $H$, and induced $H$-freeness for any path $H$ on
at most $4$ vertices or its complement (where the case of $4$ 
vertices is  proved in Section 2).

Hereditary properties which are not easily testable are
$H$-freeness for nonbipartite $H$, induced $H$-freeness for
all graphs besides the paths on at most $4$ vertices and their complements,
as well as possibly the cycle of length $4$ and its complement, perfectness
and comparability. Our techniques here can be applied to provide several
additional examples of easily testable and of non-easily testable 
hereditary properties, but most of these are somewhat artificial and
not familiar graph
properties. Does the above list  of known results suggest a 
(conjectured) characterization of all easily testable  hereditary
graph properties? At the moment we are unable to formulate such
a conjecture.

\end{document}